\documentclass[a4paper,11pt] {amsart}
\usepackage{amsthm,amssymb,latexsym,mathrsfs}
\usepackage{color}

\input amssym.def

\author{Beno\^it F. Sehba}

\title[$\Phi$-Carleson measures and multipliers]{$\Phi$-Carleson measures and multipliers between Bergman-Orlicz spaces of the unit ball of $\mathbb C^n$}
\newtheorem{theorem}{T{\hskip 0pt\footnotesize\bf HEOREM}}[section]
\newtheorem{lemma}[theorem]{L{\hskip 0pt\footnotesize\bf EMMA}}
\newtheorem{proposition}[theorem]{P{\hskip 0pt\footnotesize\bf ROPOSITION}}
\newtheorem{definition}[theorem]{D{\hskip 0pt\footnotesize\bf EFINITION}}
\newtheorem{corollary}[theorem]{C{\hskip 0pt\footnotesize\bf OROLLARY}}

\newcommand{\bprop} {\begin{proposition}}
\newcommand{\eprop} {\end{proposition}}
\newcommand{\btheo} {\begin{theorem}}
\newcommand{\etheo} {\end{theorem}}
\newcommand{\blem} {\begin{lemma}}
\newcommand{\elem} {\end{lemma}}
\newcommand{\bcor} {\begin{corollary}}
\newcommand{\ecor} {\end{corollary}}

\newcommand{\Be}{\begin{equation}}
\newcommand{\Ee}{\end{equation}}
\newcommand{\Bea}{\begin{eqnarray}}
\newcommand{\Eea}{\end{eqnarray}}
\newcommand{\Bes}{\begin{equation*}}
\newcommand{\Ees}{\end{equation*}}
\newcommand{\Beas}{\begin{eqnarray*}}
\newcommand{\Eeas}{\end{eqnarray*}}
\newcommand{\Ba}{\begin{array}}
\newcommand{\Ea}{\end{array}}

\def\C{\mathbb{C}}

\scrollmode

\begin{document}
\date{\today}
\address{Beno\^it Sehba, Department of Mathematics, University of Ghana, Legon, P.O.Box LG 62 Legon, Accra, Ghana.}
\email{bfsehba@ug.edu.gh}
\keywords{Bergman-Orlicz spaces, Carleson measures, Multipliers.}
\subjclass[2000]{Primary 47B35, Secondary 32A35, 32A37}

\begin{abstract} We define the notion of $\Phi$-Carleson measures where $\Phi$ is either a concave growth function or a convex growth function and  provide an equivalent definition. We then characterize $\Phi$-Carleson measures for Bergman-Orlicz spaces, and apply them to characterize multipliers between Bergman-Orlicz spaces.
\end{abstract}
\maketitle

\section{Introduction and statement of results}

Let denote by $d\nu$ the Lebesgue measure on the unit ball
$\mathbb B^n$ of $\mathbb C^n$ and $d\sigma$  the normalized measure on
$\mathbb S^n=\partial{\mathbb B^n}$ the boundary of
$\mathbb B^n$. By $\mathcal H(\mathbb B^n)$, we denote the
space of holomorphic functions on $\mathbb B^n.$ 

For $z=(z_1,\cdots,z_n)$ and $w=(w_1,\cdots,w_n)$ in $\C^n$, we
let ${\langle z,w\rangle=z_1\overline {w_1} + \cdots + z_n\overline {w_n}}$ so
that $|z|^2=\langle z,z\rangle=|z_1|^2 +\cdots +|z_n|^2$.

We say a function $\Phi$ is a growth function, if it is a continuous and non-decreasing function  from $[0,\infty)$ onto itself.

For $\alpha>-1$, we denote by $d\nu_{\alpha}$ the normalized Lebesgue measure $d\nu_{\alpha}(z)=c_{\alpha}(1-|z|^2)^{\alpha}d\nu(z)$, $c_\alpha$ being the normalization constant. For $\Phi$ a growth function, the weighted Bergman-Orlicz space $A_\alpha^{\Phi}(\mathbb B^n)$ is the space of holomorphic function $f$ such that
$$||f||_{\Phi,\alpha}=||f||_{A_\alpha^{\Phi}}:=\int_{\mathbb B^n}\Phi(|f(z)|)d\nu_{\alpha}(z)<\infty.$$
We define on $A_\alpha^{\Phi}(\mathbb B^n)$ the following (quasi)-norm
\begin{equation}\label{BergOrdef1}
||f||^{lux}_{\Phi,\alpha}=||f||^{lux}_{A_\alpha^{\Phi}}:=\inf\{\lambda>0: \int_{\mathbb B^n}\Phi(\frac{|f(z)|}{\lambda})d\nu_{\alpha}(z)\le 1\}
\end{equation}
which is finite for $f\in A_\alpha^{\Phi}(\mathbb B^n)$ (see \cite{sehbastevic}).


The usual weighted Bergman spaces $A_\alpha^{p}(\mathbb B^n)$ correspond to $\Phi(t)=t^p$ and are defined by
$$||f||_{p,\alpha}^p:= \int_{\mathbb B^n}|f(z)|^pd\nu_{\alpha}(z)<\infty.$$


 For $0<p<\infty$, the usual Hardy space
  $\mathcal H^p(\mathbb B^n)$,  is
 the space of all $f\in \mathcal H(\mathbb B^n)$ such that $$||f||_{p}^{p} := \sup_{0<r<1}\int_{\mathbb S^n}|f(r\xi)|^{p}d\sigma(\xi) <
 \infty.$$

Two growth functions $\Phi_1$ and $\Phi_2$ are said equivalent if there exists some constant $c$ such that
$$c\Phi_1(ct)\leq \Phi_2(t)\leq c^{-1}\Phi_1(c^{-1}t).$$
Such equivalent growth functions define the same Orlicz space.

For any $\xi\in
 \mathbb S^n$ and $\delta > 0$, the Carleson tube $Q_{\delta}(\xi)$ is defined by 
 $$Q_{\delta}(\xi)=\{z\in \mathbb B^n : |1-\langle z,\xi\rangle|< \delta\}.$$ 

 Let $\mu$ be a positive Borel measure on ${\mathbb B^n}$, and $0<s<\infty$. We say $\mu$ is a $s$-Carleson measure on $\mathbb B^n$ if there exists a constant $C$ such that for any $\xi\in \mathbb S^n$ and any $0<\delta<1$,
\begin{equation}\label{eq:Carlmeadef}
\mu(Q_\delta(\xi))\le C\delta^{ns}.
\end{equation}
When $s=1$, the above measures are called Carleson measures. Carleson measures were first introduced in the unit disk of the complex plane $\mathbb{C}$ by L. Carleson \cite{carleson1, carleson2}. These measures are pretty adapted to the studies of various questions on Hardy spaces. In his work, Carleson obtained that a measure $\mu$ is a Carleson measure if and only if the Hardy space $H^p$ embedds continuously into the Lebesgue space $L^{p}(d\mu)$. For $s>1$, again in the unit disk, Peter Duren (\cite{duren}) has proved that a measure $\mu$ is a $s$-Carleson measure if and only if the Hardy space $H^p$ embedds continuously into the Lebesgue space $L^{ps}(d\mu)$. The Characterizations of Carleson measures for Hardy spaces of the unit ball can be found in \cite{hormander, power}. These characterizations can be summarized as follows.
\begin{theorem}\label{thm:charactcarlhardy}
Let $\mu$ be a positive measure on $\mathbb{B}^n$, and $s>0$. Then the following assertions are equivalent.
\begin{itemize}
\item[(a)]There exists a constant $C_1>0$ such that for any $\xi\in \mathbb S^n$ and any $0<\delta<1$,
$$
\mu(Q_\delta(\xi))\le C_1\delta^{ns}.
$$
\item[(b)] There exists a constant $C_2>0$ such that $$\int_{\mathbb{B}^n}\frac{(1-|a|^2)^{ns}}{|1-\langle a,z\rangle|^{2ns}}d\mu(z)\leq C_2$$
for all $a\in \mathbb{B}^n$.

If moreover $s\geq 1$, then the above assertions are both equivalent to the following.
\item[(c)] There exists a constant $C_3>0$ such that for any $f\in H^p(\mathbb{B}^n)$,
\begin{equation}\label{eq:hardyembed}\int_{\mathbb{B}^n}|f(z)|^{ps}d\mu(z)\leq C_3\|f\|_{p}^s.\end{equation}
\end{itemize}
\end{theorem} 
The characterization of measures satisfying (\ref{eq:hardyembed}) with $0<s<1$ in the setting of the unit disk is due to  V. Videnskii (\cite{videnskii}). The extension of the results of L. Carleson and P. Duren to the setting of Bergman spaces of the unit disk is due to W. Hastings (\cite{hastings}) and Bergman spaces version of the result of V. Videnskii is due to D. Luecking (\cite{luecking1}). The extensions of the latter results to the unit ball are due to Cima and Wogen \cite{CW} and D. Luecking \cite{luecking2}. 

Theorem \ref{thm:charactcarlhardy} translates as follows for Bergman spaces.
\begin{theorem}\label{thm:charactcarlbergman}
Let $\mu$ be a positive measure on $\mathbb{B}^n$, $s>0$, and $\alpha>-1$. Then the following assertions are equivalent.
\begin{itemize}
\item[(a)]There exists a constant $C_1>0$ such that for any $\xi\in \mathbb S^n$ and any $0<\delta<1$,
$$
\mu(Q_\delta(\xi))\le C_1\delta^{(n+1+\alpha)s}.
$$
\item[(b)] There exists a constant $C_2>0$ such that $$\int_{\mathbb{B}^n}\frac{(1-|a|^2)^{(n+1+\alpha)s}}{|1-\langle a,z\rangle|^{2(n+1+\alpha)s}}d\mu(z)\leq C_2$$
for all $a\in \mathbb{B}^n$.

If moreover $s\geq 1$, then the above assertions are both equivalent to the following.
\item[(c)] There exists a constant $C_3>0$ such that for any $f\in A_\alpha^p(\mathbb{B}^n)$,
\begin{equation}\label{eq:bergmanembed}\int_{\mathbb{B}^n}|f(z)|^{ps}d\mu(z)\leq C_3\|f\|_{p,\alpha}^s.\end{equation}
\end{itemize}
\end{theorem}

Let us observe that (\ref{eq:Carlmeadef}) can be read as $$\mu(Q_\delta(\xi))\leq \frac{C}{\Phi(\frac{1}{\delta^n})}$$ where $\Phi(t)=t^s$. Our aim is then to obtain characterization of such measures when power functions are replaced by appropriate growth functions.
\vskip .2cm
We recall that the growth function $\Phi$ is of upper type  if we can find $q > 0$ and $C>0$ such that, for $s>0$ and $t\ge 1$,
\begin{equation}\label{uppertype}
 \Phi(st)\le Ct^q\Phi(s).\end{equation}
We denote by $\mathscr{U}^s$ the set of growth functions $\Phi$ of upper type $s$, (with $s\ge 1$), such that the function $t\mapsto \frac{\Phi(t)}{t}$ is non-decreasing. We write $$\mathscr{U}=\bigcup_{s\geq 1}\mathscr{U}^s.$$

We also recall that $\Phi$ is of lower type  if we can find $p > 0$ and $C>0$ such that, for $s>0$ and $0<t\le 1$,
\begin{equation}\label{lowertype}
 \Phi(st)\le Ct^p\Phi(s).\end{equation}
We denote by $\mathscr{L}_s$ the set of growth functions $\Phi$ of lower type $s$,  (with $s\le 1$), such that the function $t\mapsto \frac{\Phi(t)}{t}$ is non-increasing. We write $$\mathscr{L}=\bigcup_{0<s\leq 1}\mathscr{L}_s.$$

 
Note that we may always suppose that any $\Phi\in \mathscr{L}$ (resp. $\mathscr{U}$),  is concave (resp. convex) and
that $\Phi$ is a $\mathscr{C}^1$ function with derivative $\Phi'(t)\backsimeq \frac{\Phi(t)}{t}$.
\begin{definition} Let $\mu$ be a positive measure on $\mathbb{B}^n$, and let $\Phi\in \mathscr{L}\cup\mathscr{U}$. We say $\mu$ is a $\Phi$-Carleson measure if there exists a constant $C>0$ such that for any $\xi\in \mathbb S^n$ and any $0<\delta<1$,
\begin{equation}\label{eq:PhiCarlmeadef}
\mu(Q_\delta(\xi))\le \frac{C}{\Phi(\frac{1}{\delta^n})}.
\end{equation}
\end{definition}

The following result provides an equivalent definition of $\Phi$-Carleson measures.
\begin{theorem}\label{thm:main1}
Let $\mu$ be a positive measure on $\mathbb{B}^n$, and let $\Phi\in \mathscr{L}\cup\mathscr{U}$. Then the following assertions are equivalent.
\begin{itemize}
\item[(i)] $\mu$ is a $\Phi$-Carleson measure.
\item[(ii)] There exists a constant $C>0$ such that for any $a\in \mathbb{B}^n$,
\begin{equation}\label{eq:mainequiv1}\int_{\mathbb{B}^n}\Phi\left(\frac{(1-|a|^2)^{n}}{|1-\langle a,z\rangle|^{2n}}\right)d\mu(z)\leq C.\end{equation}
\end{itemize}
\end{theorem}  
The following extends Theorem \ref{thm:charactcarlbergman} to Bergman-Orlicz spaces.

\begin{theorem}\label{thm:main2}
Let $\mu$ be a positive measure on $\mathbb{B}^n$, and $\alpha>-1$. Let $\Phi_1,\Phi_2\in\mathscr{L}\cup\mathscr{U}$. Then the following assertions are equivalent.
\begin{itemize}
\item[(a)]There exists a constant $C_1>0$ such that for any $\xi\in \mathbb S^n$ and any $0<\delta<1$,
\begin{equation}\label{eq:phialphacarlmea}
\mu(Q_\delta(\xi))\le \frac{C_1}{\Phi_2\circ\Phi_1^{-1}(\frac{1}{\delta^{n+1+\alpha}})}.
\end{equation}
\item[(b)] There exists a constant $C_2>0$ such that $$\int_{\mathbb{B}^n}\Phi_2\left(\Phi_1^{-1}\left(\frac{1}{(1-|a|^2)^{n+1+\alpha}}\right)\frac{(1-|a|^2)^{2(n+1+\alpha)}}{|1-\langle a,z\rangle|^{2(n+1+\alpha)}}\right)d\mu(z)\leq C_2$$
for all $a\in \mathbb{B}^n$.

If moreover $\frac{\Phi_2}{\Phi_1}$ is non-decreasing , then the above assertions are both equivalent to the following.
\item[(c)] There exists a constant $C_3>0$ such that for any $f\in A_\alpha^{\Phi_1}(\mathbb{B}^n)$ with $\|f\|_{\Phi_1,\alpha}^{lux}\neq 0$,
\begin{equation}\label{eq:bergmanembed}\int_{\mathbb{B}^n}\Phi_2\left(\frac{|f(z)|}{C_3\|f\|_{\Phi_1,\alpha}^{lux}}\right)d\mu(z)\leq 1.\end{equation}
\end{itemize}
\end{theorem}

We call a measure satisfying (\ref{eq:phialphacarlmea}) a $(\Phi_2\circ \Phi_1^{-1},\alpha)$-Carleson measure. If a measure $\mu$ satisfies (\ref{eq:bergmanembed}), then we say $\mu$ is a $\Phi_2$-Carleson measure for $A_\alpha^{\Phi_1}(\mathbb{B}^n)$.
\vskip .2cm
Note that in \cite{Charpentier} and in \cite{Charpentiersehba}, it is proved that (\ref{eq:bergmanembed}) holds if and only if there exists $\delta_0$ such that for any $\delta\in (0,\delta_0)$, $$\mu(Q_\delta(\xi))\le \frac{C_1}{\Phi_2\circ\Phi_1^{-1}(\frac{1}{\delta^{n+1+\alpha}})}.$$ Moreover, the proof in both papers uses among others, a maximal function characterization of Bergman-Orlicz spaces. Here, we provide a proof which can be viewed as a generalization of the classical proof in the power functions case (see for example \cite{Ueki}).
\vskip .2cm
Let $X$ and $Y$ be two analytic function spaces which are metric spaces, with respective metrics $d_X$ and $d_Y$. We say that an analytic function $g$ is a multiplier from $X$ to $Y$, if there exists a constant $C>0$ such that for any $f\in X$, $$d_Y(fg,0)\leq Cd_X(f,0).$$
We denote by $\mathcal{M}(X,Y)$ the set of multipliers from $X$ to $Y$. The question of multipliers between Bergman spaces has been considered in \cite{Att, Axler1, Axler2, luecking, Vukotic, Zhao}. In particular, K. R. Attele obtained the characterization of multipliers between unweighted Bergman spaces of the unit disc of the complex plane in \cite{Att}, while the case of weighted Bergman spaces of the same setting was handled by R. Zhao in \cite{Zhao}. The proofs in \cite{Zhao} heavily make use of Carleson measures for Bergman spaces. We also use here our characterization of Carleson measures to extend the result of \cite{Zhao} on multipliers between the Bergman spaces $A_\alpha^p(\mathbb{B}^n)$ and $A_\beta^q(\mathbb{B}^n)$ with $0<p\leq q<\infty$ to a corresponding situation for Bergman-Orlicz spaces.
\vskip .2cm
Let us introduce two subsets of growth functions. We say a growth function $\Phi\in \mathscr U^q$ belongs to $\tilde{\mathscr U}$, if 
\begin{itemize}
\item[($a_1$)] There exists a constant $C_1>0$ such that for any $0<s\leq 1$, and any $t\geq 1$,
\begin{equation}\label{eq:uppertypecondmulti1}
\Phi(st)\leq C_1\Phi(s)\Phi(t).
\end{equation}
\item[($a_2$)]  There exists a constant $C_2>0$ such that for any $a,b\geq 1$, 
\begin{equation}\label{eq:uppertypecondmulti2}
\Phi\left(\frac{a}{b}\right)\leq C_2\frac{\Phi(a)}{b^q}.
\end{equation}
\end{itemize}
As examples of functions in $\tilde{\mathscr U}$, we have power functions, and for nontrivial examples, we have the functions $t\mapsto t^q\log^\alpha(C+t)$, where $q\geq 1$, $\alpha>0$ and the constant $C>0$ is large enough. 
\vskip .2cm

We say a growth function $\Phi\in \mathscr L^p$ belongs to $\tilde{\mathscr L}$, if $\Phi$ satisfies condition (\ref{eq:uppertypecondmulti1}) and if there exists a constant $C_3>0$ such that for any $a,b\geq 1$, 
\begin{equation}\label{eq:lowertypecondmulti}
\Phi\left(\frac{a}{b}\right)\leq C_3\frac{a^p}{\Phi(b)}.
\end{equation}
Clearly, power  functions are in $\tilde{\mathscr L}$. For nontrivial examples, we have the functions $t\mapsto t^p\log^\alpha(C+t)$, where $0<p\leq 1$, $\alpha<0$ and the constant $C>0$ is large enough.  To see that the latter satisfies (\ref{eq:uppertypecondmulti1}), use that if $\Phi\in \mathscr L^p$, then for any $t\geq 1$, $t^p\leq C\Phi(t)$, with $C$ the constant in (\ref{lowertype}).
\vskip .2cm
Before stating our result on multipliers of Bergman-Orlicz spaces, let us introduce another space of analytic functions. Let $\omega:(0,1]\longrightarrow (0,\infty)$. An analytic function $f$ in $\mathbb B^n$ is said to be in $\mathcal H_\omega^\infty(\mathbb B^n)$ if
\begin{equation}\label{OmegaHinfdef}
||f||_{\mathcal H_\omega^\infty}:=\sup_{z\in \mathbb B^n}\frac{|f(z)|}{\omega(1-|z|)}<\infty.
\end{equation}
Clearly, $\mathcal H_\omega^\infty(\mathbb B^n)$ is a Banach space.
\begin{theorem}\label{thm:main3}
Let $\Phi_1\in \mathscr L\cup \mathscr U$ and $\Phi_2\in \tilde{\mathscr L}\cup \tilde{\mathscr U}$. Assume that $\frac{\Phi_2}{\Phi_1}$ is nondecreasing. Let $\alpha, \beta>-1$ and define for $1\leq t<\infty$, the function
$$\gamma(t)=\frac{\Phi_2^{-1}(t^{n+1+\beta})}{\Phi_1^{-1}(t^{n+1+\alpha})}.$$
Then the following assertions hold.
\begin{itemize}
\item[(i)] If $\gamma$ is nondecreasing on $[1,\infty)$, then $$\mathcal{M}\left(A_\alpha^{\Phi_1}(\mathbb{B}^n),A_\beta^{\Phi_2}(\mathbb{B}^n)\right)=H_\omega^\infty(\mathbb{B}^n),\,\,\, \omega(s)=\gamma(\frac{1}{s}).$$
\item[(ii)] If $\gamma$ is equivalent to $1$, then $\mathcal{M}\left(A_\alpha^{\Phi_1}(\mathbb{B}^n),A_\beta^{\Phi_2}(\mathbb{B}^n)\right)=H^\infty(\mathbb{B}^n)$.
\item[(iii)] If $\gamma$ is non-increasing on $[1,\infty)$ and $\lim_{t\rightarrow \infty}\gamma(t)=0$, then $$\mathcal{M}\left(A_\alpha^{\Phi_1}(\mathbb{B}^n),A_\beta^{\Phi_2}(\mathbb{B}^n)\right)=\{0\}.$$
\end{itemize}
\end{theorem}
\vskip .1cm
Note that if the function $\gamma$ in the last theorem is non-increasing and $\lim_{t\rightarrow \infty}\gamma(t)\neq 0$, then $\gamma$ is equivalent to $1$.
\vskip .2cm
The proofs of the first two theorems will be given in section 3 and applications to embeddings between Bergman-Orlicz spaces and the pointwise multipliers of Bergman-Orlicz spaces will be provided in section 4. In the next section, we recall some results from the literature needed in our proofs.

 All over the text, we assume without loss of generality that our growth functions $\Phi$ are such that $\Phi(1)=1$. Finally, all over the text, $C$ will be a constant not necessary the same at each occurrence. We recall that
given two positive quantities $A$ and $B$, the notation  $A\lesssim B$ means that $A\le CB$ for some positive constant $C$.
  When $A\lesssim B$ and $B\lesssim A$, we write $A\backsimeq B$.
\section{Some preliminaries}
We provide in this section some useful results which are mostly related related to  growth functions. 

For $a\in \mathbb B^n$, $a\ne 0$, let $\varphi_a$ denote
the automorphism of $\mathbb B^n$ taking $0$ to $a$ defined
by
$$\varphi_{a}(z)=\frac{a - P_{a}(z) - (1 - |z|^{2})^{\frac{1}{2}}Q_{a}(z)}{1 -
 \langle z,a\rangle}$$ where $P_a$ is the projection of $\C^n$ onto the
 one-dimensional subspace span of $a$ and $Q_a=I - P_a$ where
 $I$ is the identity.
 It is easy to see that
 $$\varphi_a(0)=a,\,\,\,
 \varphi_a(a)=0,\,\,\, \varphi_a o \varphi_a(z)=z,$$ $$1-
 |\varphi_a(z)|^2 =\frac{(1-|a|^2)(1-|z|^2)}{|1-\langle z,a\rangle|^2}.$$ 

For $0<r<1$, and $a\in \mathbb B^n$, we write $r\mathbb B^n:=\{z\in \mathbb B^n: |z|<r\}$, and define the (pseudo-hyperbolic metric) ball $\Delta(a,r)$ by
$$\Delta(a,r)=\{z\in \mathbb B^n: |\varphi_a(z)|<r\}.$$  Clearly,
$\Delta(a,r)=\varphi_a(r\mathbb B^n)$. One easily check the following (for details see \cite{Ueki}).
\begin{lemma}\label{deltacarlregion}
For any $a\in \mathbb B^n$ and $0<r<1$, there exist $\xi\in \mathbb S^n$ and $\delta>0$ such that $\Delta(a,r)\subset Q_\delta(\xi)$. Moreover, $\delta\backsimeq 1-|z|^2$.
\end{lemma}

We have the following estimate (see \cite{sehbastevic, sehbatchoundja}).
\begin{lemma}\label{pointwiseberg}
Let $\Phi\in \mathscr {L}\cup \mathscr U$, $-1<\alpha<\infty$. There is a constant $C>0$ such that for any $f\in \mathcal A_{\alpha}^{\Phi}(\mathbb B^n)$, and any $a\in \mathbb B^n$,
\begin{equation}\label{pointwisebergequat}|f(a)|\le C\Phi^{-1}\left(\frac{1}{(1-|a|^2)^{n+1+\alpha}}\right)||f||_{\Phi, \alpha}^{lux}.\end{equation}
\end{lemma}


The next lemma provide a useful function in $\mathcal A_{\alpha}^{\Phi_p}(\mathbb B^n)$ (see \cite{sehbatchoundja}).
\begin{lemma}\label{testfunctionberg}
Let $-1<\alpha<\infty$, $a\in \mathbb B^n$. Suppose that $\Phi\in \mathscr {L}\cup\mathscr U$. Then the following function is in $\mathcal A_{\alpha}^{\Phi}(\mathbb B^n)$
$$f_a(z)=\Phi^{-1}\left(\frac{1}{(1-|a|)^{n+1+\alpha}}\right)\left(\frac{1-|a|^2}{1-\langle z,a\rangle}\right)^{2(n+1+\alpha)}.$$
Moreover, $||f_a||_{\Phi, \alpha}^{lux}\lesssim 1.$
\end{lemma}


\section{Proof of the theorems}
\begin{proof}[Proof of Theorem \ref{thm:main1}] The proof follows the same idea as in the power functions case (see for example \cite{ZZ}). We provide details here.

$(i)\Rightarrow (ii):$ For $|a|\le \frac{3}{4}$,
(\ref{eq:mainequiv1}) is obvious since the measure is
necessarily finite. Let $\frac{3}{4} < |a| <1$ and choose
$\xi=a/|a|$. For any nonnegative integer $k$, let $r_k =
2^{k-1}(1-|a|)$, $k=1,2,\cdots,N$ and $N$ is the smallest
integer such that $2^{N-2}(1-|a|)\ge 1$.  Let $E_1=Q_{r_1}$ and $E_k
=Q_{r_k}(\xi)-Q_{r_{k-1}}(\xi)$, $k\ge 2.$ We have
$$\mu(E_k)\le \mu(Q_{r_k}(\xi))\le
\frac{C}{\Phi\left(\frac{1}{2^{(k-1)n}(1-|a|)^n}\right)}.$$
Moreover, if $k\ge 2$ and $z\in E_k$, then \Beas |1-\langle
a,z\rangle | &=&
|1-|a|+|a|(1-\langle \xi,z\rangle )|\\ &\ge& -(1-|a|)+|a||1-\langle \xi,z\rangle |\\
&\ge& \frac{3}{4}2^{k-1}(1-|a|)-(1-|a|)\\ &\ge&
2^{k-2}(1-|a|).\Eeas
We also have for $z\in E_1$, $$|1-\langle z,a\rangle |\ge
1-|a|>\frac{1}{2}(1-|a|).$$
Let us put $$K_a(z)=\frac{(1-|a|^2)^n}{|1-\langle a,z\rangle
|^{2n}}.$$
Using the above estimates and putting $\varepsilon=1$ if $\Phi\in \mathscr U$, and $\varepsilon=p$ if $\Phi\in \mathscr L$ is of lower type $0<p<1$,  we
obtain \Beas L &:=& \int_{\mathbb{B}^n}\Phi\left(\frac{(1-|a|^2)^{n}}{|1-\langle a,z\rangle|^{2n}}\right)d\mu(z)\\ &=&\int_{\mathbb B^n}\Phi(K_{a}(z))d\mu(z)\\ &\le&
C\sum_{k=1}^{N}\frac{\Phi\left(\frac{1}{2^{2(k-2)n}(1-  |a|)^n}\right)}{\Phi\left(\frac{1}{2^{(k-1)n}(1-|a|)^n}\right)}\\
&\le&
C\sum_{k=1}^{N}\frac{1}{2^{kn\varepsilon}}<\tilde{C}<\infty.\Eeas

$(ii)\Rightarrow (i):$ Let $a\neq 0$, $a\in \mathbb B^n$. Set $\delta=1-|a|^2$ and $\xi=a/|a|$. We remark that for $z\in Q_\delta(\xi)$,
$|1-\langle z,a\rangle|\le 2(1-|a|^2)$. Hence, using (\ref{uppertype}) for $\Phi\in \mathscr U$ and the fact that the function $\frac{\Phi(t)}{t}$ is nonincreasing for $\Phi\in \mathscr L$, we obtain
\begin{eqnarray*}
\mu(Q_\delta(\xi))\Phi(\frac{1}{\delta^n}) &\lesssim& \int_{\mathbb{B}^n}\Phi\left(\frac{(1-|a|^2)^{n}}{|1-\langle a,z\rangle|^{2n}}\right)d\mu(z)\\ &\le& C.
\end{eqnarray*}
The proof is complete.
\end{proof}

\begin{proof}[Proof of Theorem \ref{thm:main2}] We observe that the implication $(b)\Rightarrow (a)$ follows the same way as in the proof of the implication $(ii)\Rightarrow (i)$ in Theorem \ref{thm:main1}. We will then only prove that $(a)\Rightarrow (c)\Rightarrow (b)$.

$(a)\Rightarrow (c)$: We fix $\frac{1}{2}<r<1$ and $z\in \mathbb B^n$. We recall that
by Lemma \ref{deltacarlregion}, $\Delta(z,r)\subset Q_\delta(\xi)$ for some $\xi\in \mathbb S^n$ and $\delta >0$ with $\delta\backsimeq 1-|z|^2$. Under $(a)$, this implies that
\begin{equation}\label{muDeltaQ}
\mu(\Delta(z,r))\le \mu(Q_\delta(\xi))\leq \frac{C_1}{\Phi_2\circ \Phi_1^{-1}\left(\frac{1}{(1-|z|^2)^{n+1+\alpha}}\right)}.
\end{equation}

 Next,  we recall that  if $\Phi\in \mathscr L$ with lower type $0<p<1$, then the growth $\Phi_p(t)=\Phi(t^{1/p})$ belongs to $\mathscr U$ (see \cite{sehbatchoundja}). We will also use the notation $\Phi_p$ for $\Phi\in \mathscr U$ noting that in this case $p=1$. 
As $|f|^p$ ($0<p\leq 1$) is $\mathcal{M}-subharmonic$, we obtain using inequality (4.3) of \cite{Stoll} and the convexity of $\Phi_p$, that
\begin{eqnarray*}
\Phi(|f(z)|) &=& \Phi_p(|f(z)|^p)\\ &\le& C_2\int_{\Delta(z,\frac{1}{2})}\Phi_p(|f(w)|^p)(1-|w|^2)^{-(n+1+\alpha)}d\nu_\alpha(w)\\ &=& C_2\int_{\Delta(z,\frac{1}{2})}\Phi(|f(w)|)(1-|w|^2)^{-(n+1+\alpha)}d\nu_\alpha(w).
\end{eqnarray*}

Let $K_1=\max\{1,C_1C_2,(CC_1C_2)^{1/p}\}$, where $C$ is the constant in (\ref{lowertype}), $C_1$ the constant in (\ref{muDeltaQ}) and $C_2$ the constant in the last inequality. Put $K=\max\{C', K_1\}$ where $C'$ is the constant in (\ref{pointwisebergequat}). Using Lemma \ref{pointwiseberg}, Fubini's lemma , estimate (\ref{muDeltaQ}), and that the function $\frac{\Phi_2}{\Phi_1}$ is nondecreasing, we obtain
\begin{eqnarray*}
L &:=& \int_{\mathbb B^n}\Phi_2\left(\frac{|f(z)|}{K||f||_{\Phi_1,\alpha}^{lux}}\right)d\mu(z)\\ &\leq& C_2\int_{\mathbb B^n}d\mu(z) \int_{\Delta(z,\frac{1}{2})}\Phi_2\left(\frac{|f(w)|}{K||f||_{\Phi_1,\alpha}^{lux}}\right)(1-|w|^2)^{-(n+1+\alpha)}d\nu_\alpha(w)\\ &\leq& C_2\int_{\mathbb B^n}\left(\int_{\mathbb B^n}\chi_{\Delta(z,\frac{1}{2})}(w)d\mu(z)\right) \Phi_2\left(\frac{|f(w)|}{K||f||_{\Phi_1,\alpha}^{lux}}\right)(1-|w|^2)^{-n-1}d\nu(w)\\ &\leq&
C_2\int_{\mathbb B^n}\Phi_2\left(\frac{|f(w)|}{K||f||_{\Phi_1,\alpha}^{lux}}\right)(1-|w|^2)^{-n-1}\mu(\Delta(w,r))d\nu(w)\\
&\leq& C_2\int_{\mathbb B^n}\Phi_1\left(\frac{|f(w)|}{K||f||_{\Phi_1,\alpha}^{lux}}\right)\frac{\Phi_2\left(\frac{|f(w)|}{K||f||_{\Phi_1,\alpha}^{lux}}\right)}{\Phi_1\left(\frac{|f(w)|}{K||f||_{\Phi_1,\alpha}^{lux}}\right)}(1-|w|^2)^{-n-1}\mu(\Delta(w,r))d\nu(w)\\
&\leq& C_2\int_{\mathbb B^n}\Phi_1\left(\frac{|f(w)|}{K||f||_{\Phi_1,\alpha}^{lux}}\right)\frac{\Phi_2\circ \Phi_1^{-1}\left(\frac{1}{(1-|w|^2)^{n+1+\alpha}}\right)}{\Phi_1\circ \Phi_1^{-1}\left(\frac{1}{(1-|w|^2)^{n+1+\alpha}}\right)}\\ & & (1-|w|^2)^{-n-1}\mu(\Delta(w,r))d\nu(w)\\ &\leq& C_1C_2\int_{\mathbb B^n}\Phi_1\left(\frac{|f(w)|}{K||f||_{\Phi_1,\alpha}^{lux}}\right)d\nu_\alpha(w)\\ &\leq& \int_{\mathbb B^n}\Phi_1\left(\frac{|f(w)|}{||f||_{\Phi_1,\alpha}^{lux}}\right)d\nu_\alpha(w)\le 1.
\end{eqnarray*}
Where we used the fact that $\chi_{\Delta(z,\frac{1}{2})}(w)\le \chi_{\Delta(w,r)}(z)$ for each $z\in \mathbb B^n$ and $w\in \mathbb B^n$. 

$(c)\Rightarrow (b)$: Let $a\in \mathbb{B}^n$. Recall with Lemma \ref{testfunctionberg} that the function $$f_a(z)=\Phi_1^{-1}\left(\frac{1}{(1-|a|)^{n+1+\alpha}}\right)\left(\frac{1-|a|^2}{1-\langle z,a\rangle}\right)^{2(n+1+\alpha)}$$ is uniformly in $A_\alpha^{\Phi_1}(\mathbb{B}^n)$. Thus the implication follows by testing $(c)$ with $f_a$ and using the monotonicity of $\Phi$ or the monotonicity of the function $\frac{\Phi(t)}{t}$. 
The proof is complete.
\end{proof}

\section{Embeddings and multipliers between Bergman-Orlicz spaces}
We start this section with an embedding result between Bergman-Orlicz spaces.
\begin{theorem}\label{thm:bergorlbergorlembed}
Let $\Phi_1,\Phi_2\in \mathscr L\cup \mathscr U$, $\alpha,\beta>-1$. Assume that $\frac{\Phi_2}{\Phi_1}$ is non-decreasing. Then $A_\alpha^{\Phi_1}(\mathbb{B}^n)$ embeds continuously into $A_\beta^{\Phi_2}(\mathbb{B}^n)$ if and only if there is a constant $C>0$ such that 
\begin{equation}\label{eq:embedcond}
\Phi_1^{-1}(t^{n+1+\alpha})\leq \Phi_2^{-1}(Ct^{n+1+\beta}),\,\,\,\textrm{for}\,\,\, t\in [1,\infty).
\end{equation}
\end{theorem}
\begin{proof}
That $A_\alpha^{\Phi_1}(\mathbb{B}^n)$ embeds continuously into $A_\beta^{\Phi_2}(\mathbb{B}^n)$ is equivalent in saying that there exists a constant $C>0$ such that for every $f\in A_\alpha^{\Phi_1}(\mathbb{B}^n)$, with $||f||_{\Phi_1,\alpha}^{lux}\neq 0$,
\begin{equation}\label{eq:equivembed}
\int_{\mathbb B^n}\Phi_2\left(\frac{|f(z)|}{C||f||_{\Phi_1,\alpha}^{lux}}\right)d\nu_\beta(z)\leq 1
\end{equation}
which is equivalent by Theorem \ref{thm:main2} in saying that $\nu_\beta$ is a $(\Phi_2\circ \Phi_1^{-1},\alpha)$-Carleson measure. Hence we only have to prove that $\nu_\beta$ is a $(\Phi_2\circ \Phi_1^{-1}, \alpha)$-Carleson measure if and only if (\ref{eq:embedcond}) holds.

Let us first assume that (\ref{eq:embedcond}) holds. Then we easily obtain for any $\xi\in \mathbb{S}^n$ and any $0<\delta<1$, that
\begin{equation}\label{eq:nubetacarl}
\nu_\beta(Q_\delta(\xi)) \sim \delta^{n+1+\beta}=\frac{C}{\Phi_2\circ \Phi_2^{-1}(\frac{C}{\delta^{n+1+\beta}})}\leq \frac{C}{\Phi_2\circ \Phi_1^{-1}(\frac{1}{\delta^{n+1+\alpha}})}.
\end{equation}
That is $\nu_\beta$ is a $(\Phi_2\circ \Phi_1^{-1}, \alpha)$-Carleson measure.

Now assume that $\nu_\beta$ is a $(\Phi_2\circ \Phi_1^{-1}, \alpha)$-Carleson measure. Then by Theorem \ref{thm:main2} there exists a constant $K>0$ such that
\begin{equation}\label{eq:repphiphi}\int_{\mathbb{B}^n}\Phi_2\left(\Phi_1^{-1}\left(\frac{1}{(1-|a|^2)^{n+1+\alpha}}\right)\frac{(1-|a|^2)^{2(n+1+\alpha)}}{|1-\langle a,z\rangle|^{2(n+1+\alpha)}}\right)d\nu_\beta(z)\leq K
\end{equation}
for all $a\in \mathbb{B}^n$.
\vskip .2cm
For $a\in \mathbb{B}^n$ given, let $\delta=1-|a|^2$ and $\xi=\frac{a}{|a|}$ ($a \neq 0$). Then using the type of $\Phi_2$ or the monotonicity of $\frac{\Phi_2(t)}{t}$, we obtain from (\ref{eq:repphiphi}) that 
$$\delta^{n+1+\beta}\Phi_2\circ \Phi_1^{-1}\left(\frac{1}{\delta^{n+1+\alpha}}\right)\backsimeq \nu_\beta(Q_\delta(\xi))\Phi_2\circ \Phi_1^{-1}\left(\frac{1}{\delta^{n+1+\alpha}}\right)\leq K,$$
that is 
$$\Phi_1^{-1}\left(\frac{1}{\delta^{n+1+\alpha}}\right)\leq \Phi_2^{-1}\left(\frac{C}{\delta^{n+1+\beta}}\right)$$
for some constant $C>0$. Thus (\ref{eq:embedcond}) holds. The proof is complete.
\end{proof}

We next consider multipliers between two Bergman-Orlicz spaces. We start by the following general result.
\begin{theorem}\label{thm:multiplbergorliczberg1}
Let $\Phi_1, \Phi_2\in \mathscr L\cup \mathscr U$. Assume that $\frac{\Phi_2}{\Phi_1}$ is nondecreasing. Let $\alpha, \beta>-1$ and define for $1\leq t<\infty$, the function
$$\gamma(t)=\frac{\Phi_2^{-1}(t^{n+1+\beta})}{\Phi_1^{-1}(t^{n+1+\alpha})}.$$
Then the following assertions hold.
\begin{itemize}
\item[(i)] If $\gamma$ is equivalent to $1$, then $\mathcal{M}\left(A_\alpha^{\Phi_1}(\mathbb{B}^n),A_\beta^{\Phi_2}(\mathbb{B}^n)\right)=H^\infty(\mathbb{B}^n)$.
\item[(ii)] If $\gamma$ is nonincreasing on $[1,\infty)$ and $\lim_{t\rightarrow \infty}\gamma(t)=0$, then $$\mathcal{M}\left(A_\alpha^{\Phi_1}(\mathbb{B}^n),A_\beta^{\Phi_2}(\mathbb{B}^n)\right)=\{0\}.$$
\end{itemize}
\end{theorem}
\begin{proof}
Let us start by the proving assertion $(i)$. We first prove the sufficiency. Let us assume that $g\in H^\infty(\mathbb{B}^n)$. Then for any $f\in A_\alpha^{\Phi_1}(\mathbb{B}^n)$, we have that
\Beas
\int_{\mathbb{B}^n}\Phi_2\left(\frac{|g(z)||f(z)|}{C\|g\|_\infty\|f\|_{\Phi_1, \alpha}^{lux}}\right)d\nu_\beta(z) &\leq& \int_{\mathbb{B}^n}\Phi_2\left(\frac{|f(z)|}{C\|f\|_{\Phi_1, \alpha}^{lux}}\right)d\nu_\beta(z)\\ &\leq& 1
\Eeas
where the last inequality follows from Theorem \ref{thm:bergorlbergorlembed}, with $C$ an appropriate constant.

Now assume that $g\in \mathcal{M}\left(A_\alpha^{\Phi_1}(\mathbb{B}^n),A_\beta^{\Phi_2}(\mathbb{B}^n)\right)$. Then by Lemma \ref{pointwiseberg} , there exists a constant $C>0$ such that for any $f\in A_\alpha^{\Phi_1}(\mathbb{B}^n)$ and any $z\in \mathbb{B}^n$,
\begin{equation}\label{eq:multopointwise}
|g(z)|f(z)||=|f(z)g(z)|\leq C\Phi_2^{-1}\left(\frac{1}{(1-|z|^2)^{n+1+\beta}}\right)\|f\|_{\Phi_2.\beta}^{lux}.
\end{equation}
Taking in the above inequality $f(z)=f_a(z)=\Phi_1^{-1}\left(\frac{1}{(1-|a|)^{n+1+\alpha}}\right)\left(\frac{1-|a|^2}{1-\langle z,a\rangle}\right)^{2(n+1+\alpha)}$ with $a\in \mathbb{B}^n$ fixed, we obtain for the same constant in (\ref{eq:multopointwise}) that for any $z\in \mathbb{B}^n$,
\begin{equation}\label{eq:multipointwise1}
|g(z)|\Phi_1^{-1}\left(\frac{1}{(1-|a|)^{n+1+\alpha}}\right)\left(\frac{1-|a|^2}{|1-\langle z,a\rangle|}\right)^{2(n+1+\alpha)}\leq C\Phi_2^{-1}\left(\frac{1}{(1-|z|^2)^{n+1+\beta}}\right).
\end{equation}
Taking in particular $z=a$ in the last equation, we obtain that for any $a\in \mathbb{B}^n$,$$|g(a)|\lesssim C.$$
Hence, $g\in H^\infty(\mathbb{B}^n)$. The proof of assertion $(i)$ is complete.
\vskip .2cm
Proof of $(ii):$ It obvious that $0$ is an element of $\mathcal{M}\left(A_\alpha^{\Phi_1}(\mathbb{B}^n),A_\beta^{\Phi_2}(\mathbb{B}^n)\right)$. Next, assume that $g\in \mathcal{M}\left(A_\alpha^{\Phi_1}(\mathbb{B}^n),A_\beta^{\Phi_2}(\mathbb{B}^n)\right)$, then as above, there is a constant $C>0$ such that for any $f\in A_\alpha^{\Phi_1}(\mathbb{B}^n)$ and any $z\in \mathbb{B}^n$, (\ref{eq:multopointwise}) holds. Testing (\ref{eq:multopointwise}) with 
$f(z)=f_a(z)=\Phi_1^{-1}\left(\frac{1}{(1-|a|)^{n+1+\alpha}}\right)\left(\frac{1-|a|^2}{1-\langle z,a\rangle}\right)^{2(n+1+\alpha)}$ with $a\in \mathbb{B}^n$ fixed, and then taking $z=a$, we obtain that for any $a\in \mathbb{B}^n$,
$$|f(a)|\lesssim C\frac{\Phi_2^{-1}\left(\frac{1}{(1-|a|^2)^{n+1+\beta}}\right)}{\Phi_1^{-1}\left(\frac{1}{(1-|a|)^{n+1+\alpha}}\right)}=\gamma\left(\frac{1}{1-|a|^2}\right).
$$
From the hypotheses on the function $\gamma$, we have that the right hand side of the last inequality goes to $0$ as $|a|\rightarrow 1$. Hence $f(a)=0$ for all $a\in \mathbb{B}^n$. The proof is complete.
\end{proof}
We finish with the following restriction to target growth functions in $\tilde{\mathscr L}\cup \tilde{\mathscr U}$.
\begin{theorem}\label{thm:multiplbergorliczberg2}
Let $\Phi_1\in \mathscr L\cup \mathscr U$ and $\Phi_2\in \tilde{\mathscr U}$. Assume that $\frac{\Phi_2}{\Phi_1}$ is nondecreasing. Let $\alpha, \beta>-1$ and define for $1\leq t<\infty$, the function
$$\gamma(t)=\frac{\Phi_2^{-1}(t^{n+1+\beta})}{\Phi_1^{-1}(t^{n+1+\alpha})}.$$
Then if $\gamma$ is nondecreasing on $[1,\infty)$, then $\mathcal{M}\left(A_\alpha^{\Phi_1}(\mathbb{B}^n),A_\beta^{\Phi_2}(\mathbb{B}^n)\right)=H_\omega^\infty(\mathbb{B}^n)$, $\omega(s)=\gamma(\frac{1}{s})$.
\end{theorem}
\begin{proof}
That any function in $\mathcal{M}\left(A_\alpha^{\Phi_1}(\mathbb{B}^n),A_\beta^{\Phi_2}(\mathbb{B}^n)\right)$ is an element of $H_\omega^\infty(\mathbb{B}^n)$ can be proved following the same idea in the proof of the necessity in assertion $(ii)$ of the previous theorem. Let us prove that any $g\in H_\omega^\infty(\mathbb{B}^n)$ is an element of $\mathcal{M}\left(A_\alpha^{\Phi_1}(\mathbb{B}^n),A_\beta^{\Phi_2}(\mathbb{B}^n)\right)$.
\vskip .2cm
Using the properties of growth functions in $\tilde{\mathscr U}$, inequality (\ref{uppertype}), we obtain for $q\geq 1$ the upper-type of $\Phi_2$, $K=\max\{1,C_1C_2\}$ where $C_1$ and $C_2$ are respectively from conditions (\ref{eq:uppertypecondmulti1}) and (\ref{eq:uppertypecondmulti2}) in the definition of the class $\tilde{\mathscr U}$, and  for $C>0$ a constant which existence has to be proved,
\Beas
L &:=& \int_{\mathbb{B}^n}\Phi_2\left(\frac{|g(z)||f(z)|}{KC\|g\|_{H_\omega^\infty}\|f\|_{\Phi_1,\alpha}^{lux}}\right)d\nu_\beta(z)\\ &\leq& \int_{\mathbb{B}^n}\Phi_2\left(\frac{\Phi_2^{-1}\left(\frac{1}{(1-|z|^2)^{n+1+\beta}}\right)}{\Phi_1^{-1}\left(\frac{1}{(1-|z|^2)^{n+1+\alpha}}\right)}\frac{|f(z)|}{KC\|f\|_{\Phi_1,\alpha}^{lux}}\right)d\nu_\beta(z)\\ &\leq& C_1\int_{\mathbb{B}^n}\Phi_2\left(\frac{\Phi_2^{-1}\left(\frac{1}{(1-|z|^2)^{n+1+\beta}}\right)}{\Phi_1^{-1}\left(\frac{1}{(1-|z|^2)^{n+1+\alpha}}\right)}\right)\Phi_2\left(\frac{|f(z)|}{KC\|f\|_{\Phi_1,\alpha}^{lux}}\right)d\nu_\beta(z)\\ &\leq& C_1C_2\int_{\mathbb{B}^n}\frac{1}{(1-|z|^2)^{n+1+\beta}\left(\Phi_1^{-1}\left(\frac{1}{(1-|z|^2)^{n+1+\alpha}}\right)\right)^q}\Phi_2\left(\frac{|f(z)|}{KC\|f\|_{\Phi_1,\alpha}^{lux}}\right)d\nu_\beta(z)\\ &\leq& \int_{\mathbb{B}^n}\frac{1}{(1-|z|^2)^{n+1+\beta}\Phi_2\circ\Phi_1^{-1}\left(\frac{1}{(1-|z|^2)^{n+1+\alpha}}\right)}\Phi_2\left(\frac{|f(z)|}{C\|f\|_{\Phi_1,\alpha}^{lux}}\right)d\nu_\beta(z).
\Eeas
To conclude, we only have to prove the existence of a constant $C>0$ such that 
$$\int_{\mathbb{B}^n}\Phi_2\left(\frac{|f(z)|}{C\|f\|_{\Phi_1,\alpha}^{lux}}\right)d\mu(z)\leq 1$$ where $$d\mu(z)=\frac{1}{\Phi_2\circ\Phi_1^{-1}\left(\frac{1}{(1-|z|^2)^{n+1+\alpha}}\right)}(1-|z|^2)^{-(n+1)}d\nu(z).$$ For this, we know from Theorem \ref{thm:main2} that it is enough to prove that $\mu$ is a $(\Phi_2\circ \Phi_1^{-1}, \alpha)$-Carleson measure.
\vskip .2cm
Let $\rho(z,w)$ be the Bergman distance between $z$ and $w$. For $R>0$ and $a\in \mathbb{B}^n$, the Bergman ball centered at $a$ and of radius $R$ is the set $$D(a,R)=\{z\in \mathbb{B}^n: \rho(z,a)<R\}.$$
Following the proof of the implication $(a)\Rightarrow (c)$ in Theorem \ref{thm:main2}, we see that to prove that $\mu$ is a $(\Phi_2\circ \Phi_1^{-1}, \alpha)$-Carleson measure, it is enough to prove that for any $a\in \mathbb{B}^n$ and $R>0$, $$\mu(D(a,R))\leq C_R\frac{1}{\Phi_2\circ\Phi_1^{-1}\left(\frac{1}{(1-|a|^2)^{n+1+\alpha}}\right)}.$$
Using that for any $z\in D(a,R)$, $1-|z|^2\sim 1-|a|^2$ and that $\nu(D(a,R))\sim (1-|a|^2)^{n+1}$, we easily obtain
\Beas
\mu(D(a,R)) &=& \int_{D(a,R)}\frac{1}{\Phi_2\circ\Phi_1^{-1}\left(\frac{1}{(1-|z|^2)^{n+1+\alpha}}\right)}(1-|z|^2)^{-(n+1)}d\nu(z)\\ &\sim& \frac{1}{\Phi_2\circ\Phi_1^{-1}\left(\frac{1}{(1-|a|^2)^{n+1+\alpha}}\right)}(1-|a|^2)^{-(n+1)}\nu(D(a,R))\\ &\sim& \frac{1}{\Phi_2\circ\Phi_1^{-1}\left(\frac{1}{(1-|a|^2)^{n+1+\alpha}}\right)}.
\Eeas
The proof is complete.
\end{proof}
We also have the following result.
\begin{theorem}\label{thm:multiplbergorliczberg3}
Let $\Phi_1\in \mathscr L $ and $\Phi_2\in \tilde{\mathscr L}$. Assume that $\frac{\Phi_2}{\Phi_1}$ is nondecreasing. Let $\alpha, \beta>-1$ and define for $1\leq t<\infty$, the function
$$\gamma(t)=\frac{\Phi_2^{-1}(t^{n+1+\beta})}{\Phi_1^{-1}(t^{n+1+\alpha})}.$$
Then if $\gamma$ is nondecreasing on $[1,\infty)$, then $\mathcal{M}\left(A_\alpha^{\Phi_1}(\mathbb{B}^n),A_\beta^{\Phi_2}(\mathbb{B}^n)\right)=H_\omega^\infty(\mathbb{B}^n)$, $\omega(s)=\gamma(\frac{1}{s})$.
\end{theorem}
\begin{proof}
Again, that any function in $\mathcal{M}\left(A_\alpha^{\Phi_1}(\mathbb{B}^n),A_\beta^{\Phi_2}(\mathbb{B}^n)\right)$ is an element of $H_\omega^\infty(\mathbb{B}^n)$ can be proved following the same idea in the proof of Theorem \ref{thm:multiplbergorliczberg1}.
\vskip .2cm
Let us prove that any $g\in H_\omega^\infty(\mathbb{B}^n)$ is an element of $\mathcal{M}\left(A_\alpha^{\Phi_1}(\mathbb{B}^n),A_\beta^{\Phi_2}(\mathbb{B}^n)\right)$.
\vskip .2cm
Let $K=\max\{1,(CC'C_1C_3)^{1/p}\}$ where $C$ the constant in (\ref{lowertype}), $C'$ the constant in (\ref{uppertype}), $C_1$ the constant in (\ref{eq:uppertypecondmulti1}), and $C_3$ the constant in (\ref{eq:lowertypecondmulti}). Using the properties of growth functions in $\tilde{\mathscr L}$, the fact that as $\Phi_2\in \mathscr L^p$, $\Phi_2^{-1}\in \mathscr U^{1/p}$, and inequality (\ref{uppertype}), we obtain for $C>0$ a constant whcih existence has to be proved, that 
\Beas
L &:=& \int_{\mathbb{B}^n}\Phi_2\left(\frac{|g(z)||f(z)|}{KC\|g\|_{H_\omega^\infty}\|f\|_{\Phi_1,\alpha}^{lux}}\right)d\nu_\beta(z)\\ &\leq& C_1\int_{\mathbb{B}^n}\Phi_2\left(\frac{\Phi_2^{-1}\left(\frac{1}{(1-|z|^2)^{n+1+\beta}}\right)}{\Phi_1^{-1}\left(\frac{1}{(1-|z|^2)^{n+1+\alpha}}\right)}\right)\Phi_2\left(\frac{|f(z)|}{KC\|f\|_{\Phi_1,\alpha}^{lux}}\right)d\nu_\beta(z)\\ &\leq& C_1C_3\int_{\mathbb{B}^n}\frac{\left(\Phi_2^{-1}\left(\frac{1}{(1-|z|^2)^{n+1+\beta}}\right)\right)^p}{\Phi_2\left(\Phi_1^{-1}\left(\frac{1}{(1-|z|^2)^{n+1+\alpha}}\right)\right)}\Phi_2\left(\frac{|f(z)|}{KC\|f\|_{\Phi_1,\alpha}^{lux}}\right)d\nu_\beta(z)\\ &\leq& C_1C_3C'\int_{\mathbb{B}^n}\frac{1}{(1-|z|^2)^{n+1+\beta}\Phi_2\left(\Phi_1^{-1}\left(\frac{1}{(1-|z|^2)^{n+1+\alpha}}\right)\right)}\Phi_2\left(\frac{|f(z)|}{KC\|f\|_{\Phi_1,\alpha}^{lux}}\right)d\nu_\beta(z)\\ &\leq& \int_{\mathbb{B}^n}\Phi_2\left(\frac{|f(z)|}{C\|f\|_{\Phi_1,\alpha}^{lux}}\right)d\mu(z)
\Eeas
where $$d\mu(z)=\frac{1}{\Phi_2\circ\Phi_1^{-1}\left(\frac{1}{(1-|z|^2)^{n+1+\alpha}}\right)}(1-|z|^2)^{-(n+1)}d\nu(z).$$
Again, to conclude, we only have to prove the existence of a constant $C>0$ such that 
$$\int_{\mathbb{B}^n}\Phi_2\left(\frac{|f(z)|}{C\|f\|_{\Phi_1,\alpha}^{lux}}\right)d\mu(z)\leq 1.$$  
This follows as in the end of the proof of Theorem \ref{thm:multiplbergorliczberg2}. The proof is complete.
\end{proof}

Let us now introduce the following class of weight functions. We say a function $\omega:(0,1]\longrightarrow (0,\infty)$ belongs to the class $\Omega$, if $\omega$ is nonincreasing, $\frac{1}{\omega}$ is of some positive lower type and the function $t\mapsto t\omega(t)$ is increasing. 
\vskip .2cm

Let $f\in \mathcal H(\mathbb B^n)$. The radial derivative $Rf$ of $f$ is given by
$$Rf(z)=\sum_{j=1}^nz_j\frac{\partial f}{\partial z_j}(z).$$

An analytic function $f$ in $\mathbb B^n$ belongs to $\Lambda_\omega$ if $Rf\in \mathcal H_{t\omega}^\infty(\mathbb B^n)$, that is
$$\sup_{z\in \mathbb B^n}\frac{(1-|z|)|Rf(z)|}{\omega(1-|z|)}<\infty.$$
$\Lambda_\omega$ is a Banach space under
$$||f||_{\Lambda_\omega}:=|f(0)|+\sup_{z\in \mathbb B^n}\frac{(1-|z|)|Rf(z)|}{\omega(1-|z|)}.$$
Note that if $\omega(t)=t^{1-\lambda}$, $1\leq \lambda<\infty$, $\Lambda_\omega$ is just $\lambda$-Bloch space usually denoted $\mathcal{B}^\lambda$; $\mathcal{B}=\mathcal{B}^1$ being the Bloch space.

The following was proved in \cite{sehbastevic}.
\begin{proposition}\label{prop:sehbastevic} Suppose that $\omega\in \Omega$. Then $\mathcal H_\omega^\infty(\mathbb B^n)=\Lambda_\omega$ with equivalent norms.
\end{proposition}

It follows in particular from Proposition \ref{prop:sehbastevic} and Theorem \ref{thm:main3} that we have the following.
\begin{corollary}
Let $0<p\leq q<\infty$, $\alpha,\beta>-1$. Define $\lambda=\frac{n+1+\beta}{q}-\frac{n+1+\alpha}{p}$. Then the following assertions hold.
\begin{itemize}
\item[(i)] If $\lambda>0$, then $\mathcal{M}\left(A_\alpha^{p}(\mathbb{B}^n),A_\beta^{q}(\mathbb{B}^n)\right)=\mathcal{B}^{\lambda+1}$.
\item[(ii)] If $\lambda=0$, then $\mathcal{M}\left(A_\alpha^{p}(\mathbb{B}^n),A_\beta^{q}(\mathbb{B}^n)\right)=H^\infty(\mathbb{B}^n)$.
\item[(ii)] If $\lambda<0$, then $\mathcal{M}\left(A_\alpha^{p}(\mathbb{B}^n),A_\beta^{q}(\mathbb{B}^n)\right)=\{0\}$.
\end{itemize}
\end{corollary}
\bibliographystyle{plain}

\end{document}